\documentclass[11pt]{amsart}
\usepackage{amsfonts}
\usepackage{graphicx}
\usepackage{tabularx}
\usepackage{array}
\usepackage[usenames,dvipsnames]{color}
\usepackage{comment}
\usepackage{amsmath}
\usepackage{amsthm}
\usepackage{amssymb}
\usepackage{fullpage}
\usepackage[dvipsnames]{xcolor}
\usepackage{tikz}
\usepackage{listings}

\tikzstyle{legend_general}=[rectangle, rounded corners, thin,
                          top color= white,bottom color=lavander!25,
                          minimum width=2.5cm, minimum height=0.8cm,
                          violet]

\renewcommand{\epsilon}{\varepsilon}
\renewcommand{\phi}{\varphi}

\DeclareMathOperator{\dist}{dist}

\DeclareMathOperator{\HG}{HG}

\DeclareMathOperator{\cl}{cl}

\newtheorem{theorem}{Theorem}[section]

\newtheorem{question}[theorem]{Question}

\newtheorem{lemma}[theorem]{Lemma}

\theoremstyle{definition}

\author{Peter Bradshaw}
\address{Department of Mathematics, Simon Fraser University, Burnaby, Canada}
\email{pabradsh@sfu.ca}

\title{Hat guessing number and guaranteed subgraphs}
\begin{document}
\maketitle
\begin{abstract}
The \emph{hat guessing number} of a graph is a parameter related to the hat guessing game for graphs introduced by Winkler. In this paper, we show that graphs of sufficiently large hat guessing number must contain arbitrary trees and arbitrarily long cycles as subgraphs. More precisely, for each tree $T$, there exists a value $N = N(T)$ such that every graph that does not contain $T$ as a subgraph has hat guessing number at most $N$, and for each integer $c$, there exists a value $N' = N'(c)$ such that every graph with no cycle of length greater than $c$ has hat guessing number at most $N'$. 
\end{abstract}

\section{Introduction}
We consider the \emph{hat guessing game}, a game played on a finite graph $G$ between two parties. One party consists of a set of \emph{players}, with exactly one player occupying each vertex of $G$, and the other party consists of a single \emph{adversary}. A player at a vertex $v \in V(G)$ can see exactly those players at the neighbors of $v$, and importantly, no player can see himself. At the beginning of the game, the adversary gives each player a hat of some color chosen from the set $\{1, \dots, q\}$. Then, each player observes the hat colors of the players at neighboring vertices and guesses the color of his own hat. 
A player cannot hear the guesses of other players. The players win if at least one player correctly guesses the color of his hat; otherwise, the adversary wins. Before the game begins, the players communicate in order to devise a guessing strategy, but the players' guessing strategy is known to the adversary. It is assumed that the players always follow a deterministic guessing strategy, so that the guess of the player at each vertex $v$ is uniquely determined by the colors of the hats at the neighbors of $v$. 
If the players have a guessing strategy that guarantees at least one correct guess regardless of the adversary's hat assignment, then we say that the hat guessing number of $G$, written $\HG(G)$, is at least $q$. In other words, $\HG(G)$ is the maximum integer $q$ for which the players have a winning guessing strategy in the hat guessing game played on $G$ with $q$ hat colors. The hat guessing game was first considered by Winkler \cite{Winkler} with hats of only two colors, and the game was first considered in its full generality by Butler, Hajiaghayi, Kleinberg, and Leighton \cite{Butler}.

We give a short and classical example from Winkler \cite{Winkler} of a winning strategy for the players in the hat guessing game. Suppose that Alice and Bob occupy the two vertices of $K_2$ and play the hat guessing game 
 against an adversary who may assign red hats and blue hats. Before the game begins, Alice agrees to guess the color of the hat that she sees on Bob's head, and Bob agrees to guess the color opposite to the color of Alice's hat. Then, when the adversary assigns hats to Alice and Bob, Alice guesses correctly if the two hats assigned have the same color, and Bob guesses correctly if the two hats assigned have different colors. Hence, Alice and Bob have a winning strategy on $K_2$ against an adversary with two hat colors, implying that $\HG(K_2) \geq 2$. In fact, Winkler \cite{Winkler} showed that $\HG(K_2) = 2$, and more generally, Feige \cite{Feige} showed that $\HG(K_n) = n$ for all $n \geq 1$.

We may alternatively characterize the hat guessing game as a graph coloring problem, as follows. Let $G$ be a graph, and let $q$ be a positive integer. For each vertex $v \in V(G)$, let $\Gamma_v$ be a function that maps each (not necessarily proper) $q$-coloring of $N(v)$ to an element of the set $\{1, \dots, q\}$; that is, $\Gamma_v:[q]^{N(v)} \rightarrow [q]$. The goal in this problem is to give $G$ a (not necessarily proper) $q$-coloring $\varphi$ such that for each vertex $v \in V(G)$, $\Gamma_v$ does not map $\phi(N(v))$ to $\phi(v)$, where $\phi(N(v))$ is the coloring of $N(v)$ under $\varphi$. If such a $q$-coloring $\phi$ exists for each set of functions $\{\Gamma_v : v \in V(G)\}$, then we say that $\HG(G) < q$. It is easy to check that in the graph coloring setting, our $q$-coloring of $G$ corresponds to a hat assignment by the adversary in the game setting, and the functions $\Gamma_v$ correspond to the individual guessing strategies of the players. Due to the natural description of the hat guessing game as a graph coloring problem, we often identify a player with the vertex she occupies, and we often refer to a hat assignment as a graph coloring. 


Determining bounds for the hat guessing numbers of large graph classes is a particularly difficult problem.
Constant upper bounds are known for the hat guessing numbers of graphs belonging to certain restricted graph classes, such as trees \cite{Butler}, graphs of bounded degree \cite{Farnik}, and graphs of bounded treedepth \cite{HeDegeneracy}. It has been asked frequently whether the hat guessing number of planar graphs is bounded above by a constant \cite{Jaroslaw, HeWindmill}, and stronger still, whether all $k$-degenerate graphs have hat guessing number bounded above by a function of $k$ \cite{Farnik, HeDegeneracy}. While both of these questions are still open, some progress was recently made toward the first question when the author \cite{PBhats} showed that outerplanar graphs have hat guessing number less than $2^{250000}$. This upper bound for planar graphs has since been dramatically reduced to $40$ by Knierim, Martinsson, and Steiner \cite{Steiner}, who considered the more general class of \emph{strongly degenerate} graphs.

In this paper, we present two more graph classes with bounded hat guessing number. We show that those graphs with a certain forbidden tree subgraph also have a constant upper bound for their hat guessing numbers, as do those graphs whose cycles are of bounded length.
 Equivalently, we may say that in a graph family $\mathcal G$ whose members have unbounded hat guessing number, every tree appears as a subgraph of some graph $G \in \mathcal G$, and for each integer $c \geq 3$, a cycle of length at least $c$ also appears as a subgraph of some graph $G \in \mathcal G$.

\begin{theorem}
\label{thm:main}
Let $T$ be a fixed tree. Then, there exists a value $N = N(T)$ such that every graph $G$ with no subgraph isomorphic to $T$ satisfies $\HG(G) \leq N$.
\end{theorem}

\begin{theorem}
\label{thm:circ}
If a graph $G$ has no cycle of length greater than $c$, then 
\[\HG(G) < \left ( \frac{64}{25} \right )^{2^{\lfloor \frac{1}{2}c^2 \rfloor -1}} + \frac{1}{2}.\]
\end{theorem}

Theorems \ref{thm:main} and \ref{thm:circ} are analogues of the well-known facts that if a graph has a forbidden tree subgraph or only has cycles of bounded lengths, then that graph has bounded chromatic number (see e.g.~\cite{Axenovich, EG}).

\section{Graphs with a forbidden tree subgraph}
In this section, we prove Theorem \ref{thm:main}. 
For our proof, we need to establish some tools. First, we consider a modified hat guessing game, introduced by Bosek et al.~\cite{Jaroslaw}, in which each player makes $s$ guesses rather than a single guess, and in which the players win if some player makes a correct guess. For this modified game on a graph $G$, we write $\HG_s(G)$ for the maximum integer $k$ such that the players have a winning strategy when each player is assigned a hat of a color from the set $\{1, \dots, k\}$. 
For the modified hat guessing game with $s$ guesses, a method of Farnik \cite{Farnik} straightforwardly implies the following lemma.

\begin{lemma}
\label{lem:LLL}
Let $s \geq 1$ be an integer.
If $G$ is a graph of maximum degree $\Delta$, then $\HG_s(G) < (\Delta + 1) es$.
\end{lemma}

We also use the following lemma of Bosek et al.~\cite{Jaroslaw}
\begin{lemma}
\label{lem:IS}
Let $s \geq 2$ be an integer. Let $G$ be a graph, and let $V(G) = A \cup B$ be a vertex partition of $G$. If each vertex of $A$ has at most $d$ neighbors in $B$, then 
$\HG(G) \leq \HG_s(G[A])$, where $s = (\HG(G[B]) + 1)^d$.
\end{lemma}

Given a rooted tree $T$ with root $r$, we say that the \emph{height} of $T$ is the maximum distance $\dist(r,v)$ taken over all vertices $v$ in $T$. Furthermore, given integers $t,h \geq 1$, we say that a \emph{$t$-ary tree of height $h$} is a rooted tree of height $h$ in which every non-leaf vertex has exactly $t$ children. Observe that every tree $T$ is a subtree of some $t$-ary tree $T'$ of some height $h$, so if $T$ is a forbidden subgraph of a graph $G$, then a fortiori, $T'$ is also a forbidden subgraph. Therefore, in order to prove Theorem \ref{thm:main}, we instead prove the following equivalent theorem.
\begin{theorem}
\label{thm:t-ary}
Let $h \geq 1$ and $t \geq 2$ be integers. There exists a value $N = N(h,t)$ such that every graph $G$ with no subgraph isomorphic to a $t$-ary tree of height $h$ satisfies $\HG(G) < N$.
\end{theorem}
\begin{proof}
We prove the theorem by induction on $h$. When $h = 1$, the theorem states that there exists a value $N$ such that every graph $G$ with no $K_{1,t}$ subgraph satisfies $\HG(G) < N$. 
Since such a graph $G$ has maximum degree $t-1$,
it follows from Lemma \ref{lem:IS} that $\HG(G) < et$. Therefore, it is sufficient to set $N(1,t) = et$.

Now, suppose that $h > 1$. Let $G$ be a graph containing no subgraph isomorphic to a $t$-ary tree of height $h$. We define $k = 2t^h$,
 and we let $A \subseteq V(G)$ be the set of all vertices in $G$ of degree less than $k$. 

We claim that $G \setminus A$ contains no subgraph isomorphic to a $t$-ary tree of height $h-1$. Indeed, suppose that $T \subseteq G \setminus A$ is a $t$-ary tree of height $h-1$. Since no leaf of $T$ was added to $A$, every leaf of $T$ has degree at least $k$ in $G$. However, this implies that a $t$-ary tree of height $h$ can be found in $G$ by greedily choosing $t$ neighbors, for each of the $t^{h-1}$ leaves of $T$, that are not in $T$ and that were not chosen as neighbors of another leaf of $T$. This is possible, since 
\[k = 2t^h >(t^{h-1} -1)t + 2t^{h-1} > (t^{h-1} -1)t + |V(T)| .\]
Since we assumed that $G$ has no $t$-ary tree of height $h$ as a subgraph, we have reached a contradiction. Therefore, we conclude that $G \setminus A$ has no subgraph isomorphic to a $t$-ary tree of height $h-1$.

Now, by the induction hypothesis, $\HG(G \setminus A) < N(h-1, t)$. 
Furthermore, by Lemma \ref{lem:LLL}, for each $s \geq 1$, $\HG_s(G[A]) < eks$. 
Therefore, letting $s = (\HG(G \setminus A) + 1)^k$, Lemma \ref{lem:IS} implies that 
\[\HG(G) \leq \HG_s(G[A]) < eks \leq ek N(h-1,t)^k .\]
Thus, 
we let $N(h,t) = ek N(h-1,t)^k = 2et^h N(h-1,t)^{2t^h} $, and then $\HG(G) < N(h,t)$. This completes induction and the proof.
\end{proof}
Note that in the proof of Theorem \ref{thm:t-ary}, the constant $N(h,t)$ grows very large very quickly due to the repeated applications of Lemma \ref{lem:IS}. We do not attempt to optimize this constant, because we suspect that our method is not best possible.

\section{Graphs with bounded circumference}
In this section, we prove Theorem \ref{thm:circ}. For a graph $G$, the \emph{circumference} of $G$ is the length of the longest cycle in $G$. Hence, by proving Theorem \ref{thm:circ}, we show that a graph with a bounded circumference also has a bounded hat guessing number. In our proof, we often consider the blocks of a connected graph, which are defined as follows. For a graph $G$, a \emph{block} of $G$ is a nonempty subgraph $H \subseteq G$ satisfying the following properties:
\begin{itemize}
\item $H$ is either $2$-connected or isomorphic to $K_2$;
\item Every connected subgraph $H' \subseteq G$ satisfying $H \subsetneq H'$ has a cut vertex.
\end{itemize}
In other words, $H$ is a block of $G$ if $H$ is maximal with respect to the property of being either $2$-connected or isomorphic to $K_2$.

Our strategy for proving Theorem \ref{thm:circ} is as follows.
First, we show that the hat guessing number of a graph $G$ can be bounded above by considering each block of $G$ individually. Then, we show that each block of a graph of bounded circumference also has bounded treedepth, and we use this fact to obtain an upper bound for $\HG(G)$. Some of our intermediate techniques, such as bounding the hat guessing number of a graph using its block decomposition, may have other applications in the study of the hat guessing game. 

 
 In some of our intermediate results, we also restrict the set of hat colors available to the adversary at certain vertices. Kokhas and Latyshev \cite{KokhasI} showed that when the hat color set available to the adversary differs between vertices, only the number of hat colors available to the adversary at each vertex affects whether or not the players have a winning strategy on $G$.

\subsection{Using block decompositions to bound hat guessing number}
In this subsection, we show that in order to bound $\HG(G)$ for a graph $G$, it is enough to bound $\HG_2(B)$ for each block $B$ of $G$, where $\HG_2(B)$ is the maximum number of hat colors with which the players in $B$ have a winning strategy in the game on $B$ with two guesses. We believe that this idea has potential for broad application in the study of the hat guessing game.

The key result (Lemma \ref{lem:blocks}) of this subsection follows from two lemmas: Lemma \ref{lem:2} and Lemma \ref{lem:rus}. The proof of Lemma \ref{lem:2} is similar to the original proof of Lemma \ref{lem:IS}, and Lemma \ref{lem:rus} comes directly from \cite{KokhasI} with minor changes. We give a full proof for each lemma, since Lemma \ref{lem:2} has a short and elegant proof, and since the original proof of Lemma \ref{lem:rus} uses the somewhat complicated notation of constructors.
\begin{lemma}
\label{lem:2}
Let $G$ be a graph, let $v \in V(G)$, and let $\ell = \HG_2(G)$. If the adversary has $2$ available hat colors at $v$ and $\ell + 1$ available hat colors at each other vertex of $G$, then the adversary has a winning strategy in the hat guessing game on $G$.
\end{lemma}
\begin{proof}
We write
 $H = G \setminus \{v\}$, and we assume without loss of generality that the hat color set available to the adversary at each vertex of $V(H)$ is $\{1, \dots, \ell+1\}$.
We fix a set of exactly $2$ colors at $v$ and a guessing strategy $\Gamma = \{\Gamma_v:v \in V(G)\}$ on $G$. For each of the two possible colors at $v$, $\Gamma$ determines a unique guessing strategy on $H$. Therefore, for each hat assignment on $H$, each vertex $u \in V(H)$ guesses one of $2$ possible colors. Since $\HG_2(H) \leq \HG_2(G) \leq \ell$, we may assign each vertex of $H$ a hat from the color set $\{1, \dots, \ell + 1\}$ in such a way that for each vertex $u \in V(H)$, neither of the two possible colors guessed at $u$ is correct. We give $V(H)$ such a hat assignment, and hence no vertex of $H$ guesses its hat color correctly. 

Now, with hat colors at $V(H)$ assigned, $\Gamma_v$ uniquely determines a guess at $v$, so we may assign $v$ a color that does not match it guess. Therefore, the adversary has a winning hat assignment with $2$ colors available at $v$ and $\ell+1$ colors available at every other vertex of $G$.
\end{proof}
\begin{lemma}
\label{lem:rus}
Let $G$ be a graph with two subgraphs $G_1$ and $G_2$ and with a cut vertex $v$, such that $G = G_1 \cup G_2$ and $V(G_1) \cap V(G_2)=  \{v\}$. If $\HG(G_1) \leq \ell$ and $\HG_2(G_2) \leq \ell$, then
$\HG(G) \leq \ell$.
\end{lemma}
\begin{proof}
We show that the adversary wins the hat guessing game on $G$ when $\ell+1$ colors are available at each vertex. Let $\Gamma = \{\Gamma_v:v\in V(G)\}$ be a fixed guessing strategy on $G$. Note that since $v$ is a cut vertex, the guess of each vertex of $V(G_1) \setminus \{v\}$ depends entirely on the colors assigned to $V(G_1)$. Since $\HG(G_1) \leq \ell$, the adversary has at least one hat assignment on $G_1$ that causes all players at $V(G_1) \setminus \{v\}$ to guess incorrectly when following $\Gamma$. Let $K$ be the set of all such hat assignments on $G_1$ that cause all players at $V(G_1) \setminus \{v\}$ to guess incorrectly.

Now, let $A$ be the set of colorings of $N_{G_1}(v)$ that can be extended to a coloring in $K$. (In other words, $A$ is the set of colorings in $K$ restricted to $N_{G_1}(v)$.) We claim that for some coloring $\alpha \in A$, there are at least two colors $\gamma$ such that the coloring $\alpha$ can be extended to a coloring $\phi \in K$ for which $\phi(v) = \gamma$. Indeed, suppose that for every $\alpha \in A$, there is a unique color $\gamma_{\alpha}$ such that $\alpha$ can be extended to an assignment $\phi \in K$ for which $\phi(v) = \gamma_{\alpha}$. Then we claim that the players have a winning strategy on $G_1$ with $\ell+1$ colors, as follows. Given a hat assignment $\phi$ on $G_1$, the players at $V(G_1) \setminus \{v\}$ follow the strategy $\Gamma$.
We write $\alpha$ for the coloring of $N_{G_1}(v)$ under $\phi$, and if $\alpha \in A$, then $v$ guesses $\gamma_{\alpha}$; otherwise, $v$ guesses the least-valued available color.
 If $\phi$ is winning for the adversary, then every player at $V(G_1) \setminus \{v\}$ guesses incorrectly, so $\phi \in K$, and $\phi$ determines a coloring $\alpha \in A$. Then, there exists a unique color $\gamma_{\alpha}$ such that $\alpha$ can be extended to a coloring in $K$ in which $\phi(v) = \gamma_{\alpha}$; hence, $\phi(v) = \gamma_{\alpha}$, and the player at $v$ guesses correctly. Since the players have a winning strategy on $G_1$ with $\ell+1$ colors, the assumption that $\HG(G_1) \leq \ell$ is contradicted. Therefore, for some coloring $\alpha \in A$, there are two colors $\gamma_1, \gamma_2$ such that the coloring $\alpha$ can be extended to colorings $\phi_1, \phi_2 \in K$ for which $\phi_1(v) = \gamma_1$ and $\phi_2(v) = \gamma_2$. 

We can now show a winning hat assignment for the adversary. First, the adversary chooses two hat assignments $\phi_1, \phi_2 \in K$ for $V(G_1) $ so that $\phi_1$ and $\phi_2$ agree on $N_{G_1}(v)$ but not on $v$. 
 We write $\phi_1(v) = \gamma_1$ and $\phi_2(v) = \gamma_2$. The adversary commits to using either $\phi_1$ or $\phi_2$ on $V(G_1)$. Next, we observe that as a coloring $\alpha$ of $N_{G_1}(v)$ has been fixed by $\phi_1$ and $\phi_2$, $\Gamma$ determines a unique guessing strategy for the vertices of $G_2$ that depends only on the hat colors at $V(G_2)$. 
We also observe that two colors are available to the adversary at $v$ and that $\ell+1$ colors are available to the adversary at each other vertex of $V(G_2) \setminus \{v\}$.
Since $\HG_2(G_2) \leq \ell$, it follows from Lemma \ref{lem:2} that
 the adversary may choose some hat assignment $\psi$ on $V(G_2)$ so that $\psi(v) \in \{\gamma_1, \gamma_2\}$ and so that all vertices of $G_2$ guess incorrectly. Writing $\psi(v) = \gamma_i$, we then see that $\psi \cup \phi_i$ is a hat assignment on $G$ that causes all players in $V(G)$ to make an incorrect guess. 

Hence, the adversary has a winning hat assignment on $G$ when $\ell+1$ vertices are available at each vertex, and thus $\HG(G) < \ell+1$. This completes the proof.
\end{proof}
Now, we are ready for the key lemma of this subsection.

\begin{lemma}
\label{lem:blocks}
Let $G$ be a graph with blocks $B_1, \dots, B_k$. If $\HG_2(B_i) \leq \ell$ for each block $B_i$, then $\HG(G) \leq \ell$.
\end{lemma}
\begin{proof}
If $G$ is $2$-connected, then by assumption, $\HG(G) \leq \HG_2(G) \leq \ell$, and we are done. Otherwise, we assume without loss of generality that $G$ is connected and induct on the number of blocks in $G$. Suppose that $G$ has $k > 1$ blocks. Let $B$ be a terminal block of $G$, that is, a block of $G$ containing a single cut vertex $v$. Let $A_1, \dots, A_t$ be maximal connected subgraphs of $G$ in which $v$ is not a cut vertex, and write $A_t = B$. Since $v$ is a cut vertex of $G$, it follows that $t \geq 2$, and $V(A_i) \cap V(A_j) = \{v\}$ for any distinct pair $i,j \in \{1, \dots, t\}$. We then let $G' = A_1 \cup \dots \cup A_{t-1}$, and $G'$ is connected with fewer than $k$ blocks.

By the induction hypothesis, $\HG(G') \leq \ell$, and by our assumption, $\HG_2(B) \leq \ell$. Thus, by Lemma \ref{lem:rus}, the hat guessing number of $G$ is at most $\ell$. This completes induction and the proof.
\end{proof}

Lemma \ref{lem:blocks} seems to have the potential for broad application in the study of the hat guessing game. For instance, one can use the lemma to prove that a cactus graph $G$ satisfies $\HG(G) \leq 16$ after observing that $\HG_2(C) < 6e$ for every cycle $C$ by Lemma \ref{lem:LLL}.

\subsection{Proof of Theorem \ref{thm:circ}}
With Lemma \ref{lem:blocks} in place, we are ready to begin the proof of Theorem \ref{thm:circ}. In our proof, we repeatedly use the notion of \emph{treedepth}, defined as follows. Given a rooted tree $T$ with root $r$, the \emph{height} of a vertex $v \in V(T)$ is the distance in $T$ from $r$ to $v$. Then, the height of $T$ is the maximum height of a vertex in $T$. We say that the \emph{closure} of $T$ is the graph on $V(T)$ obtained by adding an edge between every ancestor-descendant pair in $T$. We write $\cl(T)$ for the closure of $T$. Then, we say that a graph $G$ has treedepth at most $d$ if $G$ is a subgraph of $\cl(T)$ for some rooted tree $T$ of height $d-1$. The reason for this ``off by one" discrepancy is that in a rooted tree of height $d-1$, a maximum length path from the root to a leaf contains exactly $d$ vertices. In \cite{HeDegeneracy}, He and Li show that graphs of bounded treedepth also have bounded hat guessing number. 

First, we establish two lemmas that relate the circumference of a graph to its treedepth. The first of these lemmas is a classical lemma of Dirac that follows straightforwardly from Menger's theorem.
\begin{lemma}[\cite{DiracPath}]
\label{lem:Dirac}
If $G$ is a two-connected graph with a path of length $\ell$, then $G$ contains a cycle of length greater than $\sqrt{2 \ell}$. 
\end{lemma}
From Lemma \ref{lem:Dirac}, we see that in a two-connected graph of circumference $c$, every path has length less than $ \frac{1}{2} c^2 $. Using this fact, we prove our next lemma.
\begin{lemma}
\label{lem:td}
If $G$ is a two-connected graph of circumference $c$, then the treedepth of $G$ is at most $\frac{1}{2}c^2 $.
\end{lemma}
\begin{proof}
We choose a vertex $r \in V(G)$, and we obtain a rooted tree $T \subseteq G$ by performing a depth-first search from $r$. Since $T$ is obtained from a depth-first search, every edge of $G$ joins an ancestor-descendant pair in $T$. Hence, $G$ is a subgraph of $\cl(T)$. By Lemma \ref{lem:Dirac}, the maximum distance in $T$ from $r$ to a leaf of $T$ is at most $\frac{1}{2}c^2  - 1$, so the treedepth of $G$ is at most $\frac{1}{2}c^2 $.
\end{proof}

From Lemma \ref{lem:td}, it follows that in a graph $G$ of bounded circumference, every block $B$ of $G$ has bounded treedepth. Hence, if we show that $\HG_2(B)$ is bounded for each block $B$ of $G$, then it follows from Lemma \ref{lem:blocks} that $\HG(G)$ is bounded. 

With this in mind, we aim to show that for graphs $H$ of bounded treedepth, $\HG_2(H)$ is also bounded.
This can be achieved by following a method of He and Li \cite{HeDegeneracy}, originally used to bound the hat guessing number of graphs of bounded treedepth in the traditional hat guessing game with single guesses. He and Li consider the recursively defined sequence
\[s_{n+1} = 1 + \prod_{i = 0}^n s_n, \indent s_0 = 1,\]
which is commonly known as Sylvester's sequence, and they show that given a rooted tree $T$, in order for the adversary to defeat the players on $\cl(T)$, it is enough for the adversary to have $s_{t+1}$ colors available for each vertex at height $t$, for each $t \geq 0$. Since a graph of treedepth $d$ is a subgraph of the closure of a rooted tree $T$ whose vertices achieve a maximum height of $d - 1$, it then follows that the hat guessing number of a graph of treedepth $d$ is less than $s_d$.

For the modified game in which each player is allowed $s$ guesses, we consider the recursively defined sequence
\[a_{s,n+1} = 1 + s \prod_{i = 0}^n a_{s,i}, \indent a_{s,0} = 1.\] 
By closely following the methods of \cite{Golomb}, 
we obtain the following closed form for each value $a_{s,n}$.
\begin{lemma}
\label{lem:theta}
For each value $s \geq 1$, there exists a constant $\theta_s  \approx  1.0213 \left (s + \frac{1}{2} \right ) $ such that for all $n \geq 1$, 
\[
a_{s,n} < (\theta_s)^{2^{n-1}} + \frac{1}{2}.
\]
\end{lemma}
\begin{proof}
To simplify notation, we let $s$ be fixed, and we write $a_n = a_{s,n}$.
For each $n \geq 0$, we define $b_n = a_n - \frac{1}{2}$. Then, we observe that for all $n \geq 1$, 
\[a_n = 1 + s \prod_{i = 0}^{n-1} a_i = 1 + \frac{a_{n+1} - 1}{a_n},\]
which simplifies to $a_{n+1} - \frac{1}{2} = a_n^2 - a_n + \frac{1}{2}$. After substitution, $b_{n+1} = b_n^2 + \frac{1}{4}$.

Next, for each $n \geq 0$, we define $\theta_n = (b_{n+1})^{2^{-n}}$. Golomb \cite{Golomb} shows that the limit $\theta = \lim_{n \rightarrow \infty} \theta_n$ exists. Furthermore, for $n \geq 1$,
\[\theta_n = \left (b_n^2 + \frac{1}{4} \right )^{2^{-n}} = (b_n)^{2^{-n+1}} \left ( 1 + \frac{1}{4b_n^2} \right ) ^{2^{-n}} = \theta_{n-1} \left ( 1 + \frac{1}{4b_n^2} \right ) ^{2^{-n}}.  \]
Therefore, for all $n \geq 0$,
\[\theta_n = \theta_0 \prod_{i = 1}^n  \left ( 1 + \frac{1}{4b_i^2} \right ) ^{2^{-i}} = \left (s+\frac{1}{2} \right ) \prod_{i = 1}^n  \left ( 1 + \frac{1}{4b_i^2} \right ) ^{2^{-i}} .\]
A computation shows that $\theta = \lim_{n \rightarrow \infty} \theta_n \approx 1.0213 \left (s + \frac{1}{2} \right )$, and furthermore, $\theta_n < \theta$ for all $n \geq 0$. Therefore, for all $n \geq 1$, 
\[a_n = b_n + \frac{1}{2}  = (\theta_{n-1})^{2^{n-1}} +\frac{1}{2} < \theta^{2^{n-1}} + \frac{1}{2}.\]
\end{proof}

Now, by following the method of He and Li \cite{HeDegeneracy}, we show that a graph of treedepth $d$ has hat guessing number less than $a_{s,d}$ in the game with $s$ guesses, as follows.

\begin{theorem}
\label{thm:td}
If $G$ is a graph of treedepth $d$, then $\HG_s(G) < a_{s,d}$.
\end{theorem}
\begin{proof}
Again, to simplify notation,  we let $s$ be fixed, and we write $a_t = a_{s,t}$ for all $t \geq 0$. 
We aim to show that if $T$ is a rooted tree of height $d-1$, then $\HG_s(\cl(T)) < a_d$. In order to do this, we will actually prove a stronger result.
We will show by induction on $|V(T)|$ that the adversary has a winning hat assignment on $\cl(T)$ as long as there are $a_{t+1}$ colors available at each vertex of height $t$, for each $t \geq 0$. Since the vertices of $T$ achieve a height of at most $d-1$, this is enough to prove the theorem.

If $|V(T)| = 1$, then $T = \cl(T)$ has a single vertex $r$ at height $0$. Since $a_1 = s+1$ and since $r$ only has $s$ guesses, the adversary can assign a color to $r$ that is not guessed and win the game. 

Next, suppose that $|V(T)| > 1$. Assume that a hat guessing strategy is fixed on $\cl(T)$. For each $t \geq 0$, we fix a list of $a_{t+1}$ colors at each vertex of $T$ at height $t$. Now, let $v$ be a leaf of $T$ at height $k \geq 1$. We write $u_1, \dots, u_k$ for the descendants of $v$ at heights $0, \dots, k-1$, respectively. Since $v$ is adjacent in $\cl(T)$ exactly to $u_1, \dots, u_k$, and since the total number of colorings of the set $\{u_1, \dots, u_k\}$ is at most $P:=a_1 \dots a_{k}$, each of which determines $s$ guesses at $v$, it follows that the $s$ guesses made at $v$ comes from a set of at most $sP$ colors. However, since the adversary has $a_{k+1} = 1 + sP$ colors available at $v$, the adversary can assign $v$ a color $\gamma$ that $v$ will not guess. Hence, the adversary colors $v$ with $\gamma$, and $v$ does not guess its hat color correctly. 

Now, with $\gamma$ fixed at $v$, the players have a unique hat guessing strategy on $\cl(T \setminus \{v\})$. However, by the induction hypothesis, the adversary has a winning hat assignment on $\cl(T \setminus \{v\})$ with the available colors at each vertex. Therefore, the adversary can complete a winning coloring on $V(T)$ and win the game, and it follows that $\HG_s(\cl(T)) < a_d$.
\end{proof}

With all of these tools in place, we can now finish the proof of Theorem \ref{thm:circ}.
Let $B$ be a block of $G$. Since $B$ is two-connected, the treedepth of $B$ is at most $\lfloor \frac{1}{2}c^2 \rfloor$ by Lemma \ref{lem:td}. Therefore, by Theorem \ref{thm:td}, $\HG_2(B) < a_{2,d}$, with $d = \lfloor \frac{1}{2}c^2 \rfloor$. Then, by applying Lemma \ref{lem:blocks}, we know that $\HG(G) < a_{2,d}$. Then, by Lemma \ref{lem:theta},
$a_{2,d} < (\theta_2 )^{2^{d-1}} < \left (\frac{64}{25} \right )^{2^{d-1}}$, and 
 the proof is complete.

\section{Open questions}
While we showed that the graphs that forbid a tree $T$ as a subgraph have bounded hat guessing number, it is still open whether the graphs that forbid a general subgraph $H$ as a subgraph have bounded hat guessing number. We suspect that 
if $H$ contains a cycle, then there exists a graph $G$ with no subgraph isomorphic to $H$ for which the hat guessing number of $G$ is arbitrarily large, and we pose the following question:
\begin{question}
\label{q:H}
Is it true that the graphs with $H$ as forbidden subgraph have bounded hat guessing number if and only if $H$ is acyclic?
\end{question}
If Question \ref{q:H} has a positive answer, then it would be analagous to the well-known fact that the graphs that forbid $H$ as a subgraph have bounded chromatic number if and only if $H$ is acyclic (see e.g.~\cite{Axenovich}).
One possible way to settle Question \ref{q:H} would be to show an affirmative answer to the following question, which was asked by He, Ido, and Przybocki \cite{HeWindmill}:
\begin{question}
\label{q:girth}
Do there exist graphs with arbitrarily large girth and hat guessing number?
\end{question}
If Question \ref{q:girth} has an affirmative answer, then for any graph $H$ with at least one cycle, one can construct a family of graphs with no copy of $H$ as a subgraph and with arbitrarily large hat guessing number by considering those graphs $G$ with girth larger than the girth of $H$, which would give a positive answer to Question \ref{q:H}. 

On the other hand, if Question \ref{q:girth} has a negative answer, it would be natural to ask if any particular cycle length is necessary in a graph family of unbounded hat guessing number. For instance, must a graph of sufficiently large hat guessing number contain a $4$-cycle? Gadouleau and Georgiou \cite{GadouleauHats} have shown that $\HG(K_{n,n^n}) > n$, so evidently no odd cycle is necessary in a family of graphs of unbounded hat guessing number, but no construction exists for graphs of arbitrarily large hat guessing number that contain no $4$-cycle. In fact, based on current knowledge, graphs of sufficiently large hat guessing number may necessarily contain a cycle of every even length up to some value.

\bibliographystyle{plain}
\bibliography{HatBib}

\end{document}